\documentclass[reqno, 12pt]{amsart}
\def\ssign{\textsection\nobreak\hspace{1pt plus 0.3pt}}
\makeatletter
\let\origsection=\section 
\def\mysection{\@mystartsection{section}{1}\z@{.7\linespacing\@plus\linespacing}{.5\linespacing}{\normalfont\scshape\centering\ssign}}
\def\section{\@ifstar{\origsection*}{\mysection}}
\def\appendix{\par\c@section\z@ \c@subsection\z@
   \let\sectionname\appendixname
   \let\section=\origsection
   \def\thesection{\@Alph\c@section}} 
\def\@mystartsection#1#2#3#4#5#6{\if@noskipsec \leavevmode \fi
 \par \@tempskipa #4\relax
 \@afterindenttrue
 \ifdim \@tempskipa <\z@ \@tempskipa -\@tempskipa \@afterindentfalse\fi
 \if@nobreak \everypar{}\else
     \addpenalty\@secpenalty\addvspace\@tempskipa\fi
 \@dblarg{\@mysect{#1}{#2}{#3}{#4}{#5}{#6}}}
\def\@mysect#1#2#3#4#5#6[#7]#8{\edef\@toclevel{\ifnum#2=\@m 0\else\number#2\fi}\ifnum #2>\c@secnumdepth \let\@secnumber\@empty
  \else \@xp\let\@xp\@secnumber\csname the#1\endcsname\fi
  \@tempskipa #5\relax
  \ifnum #2>\c@secnumdepth
    \let\@svsec\@empty
  \else
    \refstepcounter{#1}\edef\@secnumpunct{\ifdim\@tempskipa>\z@ \@ifnotempty{#8}{\@nx\enspace}\else
        \@ifempty{#8}{.}{\@nx\enspace}\fi
    }\@ifempty{#8}{\ifnum #2=\tw@ \def\@secnumfont{\bfseries}\fi}{}\protected@edef\@svsec{\ifnum#2<\@m
        \@ifundefined{#1name}{}{\ignorespaces\csname #1name\endcsname\space
        }\fi
      \@seccntformat{#1}}\fi
  \ifdim \@tempskipa>\z@ \begingroup #6\relax
    \@hangfrom{\hskip #3\relax\@svsec}{\interlinepenalty\@M #8\par}\endgroup
    \ifnum#2>\@m \else \@tocwrite{#1}{#8}\fi
  \else
  \def\@svsechd{#6\hskip #3\@svsec
    \@ifnotempty{#8}{\ignorespaces#8\unskip
       \@addpunct.}\ifnum#2>\@m \else \@tocwrite{#1}{#8}\fi
  }\fi
  \global\@nobreaktrue
  \@xsect{#5}}
\makeatother

\usepackage{amsmath,amssymb,amsthm}
\usepackage{mathrsfs}
\usepackage{mathabx}\changenotsign
\usepackage{dsfont}

\usepackage{xcolor}
\usepackage[backref]{hyperref}
\hypersetup{
    colorlinks,
    linkcolor={red!60!black},
    citecolor={green!60!black},
    urlcolor={blue!60!black}
}

\usepackage[open,openlevel=2,atend]{bookmark}

\usepackage[abbrev,msc-links,backrefs]{amsrefs}

\usepackage{doi}

\renewcommand{\PrintDOI}[1]{\doi{#1}}

\usepackage[T1]{fontenc}
\usepackage{lmodern}
\usepackage[babel]{microtype}
\usepackage[english]{babel}

\usepackage{geometry}
\numberwithin{equation}{section}
\numberwithin{figure}{section}

\usepackage{enumitem}

\def\alabel{\upshape({\itshape \alph*\,})}

\theoremstyle{plain}
\newtheorem{thm}{Theorem}[section]

\newtheorem{prop}[thm]{Proposition}

\newtheorem{prob}[thm]{Problem}

\newtheorem{lemma}[thm]{Lemma}

\theoremstyle{definition}

\newtheorem{dfn}[thm]{Definition}

\let\eps=\varepsilon
\let\theta=\vartheta
\let\rho=\varrho
\let\phi=\varphi

\def\NN{\mathds N}

\def\RR{\mathds R}

\let\polishlcross=\l
\def\l{\ifmmode\ell\else\polishlcross\fi}

\makeatletter
\def\moverlay{\mathpalette\mov@rlay}
\def\mov@rlay#1#2{\leavevmode\vtop{   \baselineskip\z@skip \lineskiplimit-\maxdimen
   \ialign{\hfil$\m@th#1##$\hfil\cr#2\crcr}}}
\newcommand{\charfusion}[3][\mathord]{
    #1{\ifx#1\mathop\vphantom{#2}\fi
        \mathpalette\mov@rlay{#2\cr#3}
      }
    \ifx#1\mathop\expandafter\displaylimits\fi}
\makeatother

\newcommand{\dcup}{\charfusion[\mathbin]{\cup}{\cdot}}

\usepackage{scalerel}
\newcommand{\vrhup}[1]{\scaleobj{0.6}{\scalerel*{\rightharpoonup}{#1}}}
\newcommand{\vrlhup}[1]{\scaleobj{0.6}{\scalerel*{\leftharpoonup}{#1}}}
\newcommand{\wrhup}{\scaleobj{0.6}{\scalerel*{\rightharpoonup}{W}}}
\def\seq#1{\ThisStyle{
	\mathord{\vbox{\offinterlineskip
		\ialign{\hfil##\hfil\cr
			\noalign{\kern+2.5\scriptspace}
			$\SavedStyle{}_{\smash{\,\vrhup{#1}}}$\cr
    		\noalign{\kern-0.5\scriptspace}
    		$\SavedStyle#1$\cr}}}}}
\def\lseq#1{\ThisStyle{
	\mathord{\vbox{\offinterlineskip
		\ialign{\hfil##\hfil\cr
			\noalign{\kern+2.5\scriptspace}
			$\SavedStyle{}_{\,\smash{\vrlhup{#1}}}$\cr
    		\noalign{\kern-0.5\scriptspace}
    		$\SavedStyle#1$\cr}}}}}
\def\wseq#1{\ThisStyle{
	\mathord{\vbox{\offinterlineskip
		\ialign{\hfil##\hfil\cr
			\noalign{\kern+2.5\scriptspace}
			$\SavedStyle{}_{\smash{\,\wrhup}}$\cr
    		\noalign{\kern-0.5\scriptspace}
    		$\SavedStyle#1$\cr}}}}}

\let\vec=\seq

\let\lra=\longrightarrow
\let\to=\lra

\makeatletter
\newcommand{\pushright}[1]{\ifmeasuring@#1\else\omit\hfill$\displaystyle#1$\fi\ignorespaces}
\newcommand{\pushleft}[1]{\ifmeasuring@#1\else\omit$\displaystyle#1$\hfill\fi\ignorespaces}
\makeatother

\def\ex{\mathrm{ex}}
\def\gT{\mathfrak{T}}

\def\vpi{\vec\pi}

\begin{document}
\title[Relative Ordered Tur\'an densities]{Relative Tur\'an densities of ordered graphs}

\author[Chr. Reiher]{Christian Reiher}
\address{Fachbereich Mathematik, Universit\"at Hamburg, Hamburg, Germany}
\email{christian.reiher@uni-hamburg.de}

\author[V. R\"{o}dl]{Vojt\v{e}ch R\"{o}dl}
\address{Department of Mathematics, Emory University, Atlanta, USA}
\email{vrodl@emory.edu}

\author{Marcelo Sales}
\address{Department of Mathematics, University of California, Irvine, CA, USA}
\email{mtsales@uci.edu}

\author[M. Schacht]{Mathias Schacht}
\address{Fachbereich Mathematik, Universit\"at Hamburg, Hamburg, Germany}
\email{schacht@math.uni-hamburg.de}

\thanks{The second author is supported by NSF grant DMS~2300347 and
	the third author is supported by US Air Force grant FA9550-23-1-0298.}


\begin{abstract}
	We introduce a modification of the Tur\'an density of ordered graphs
	and investigate this graph parameter.
\end{abstract}

\maketitle

\section{Introduction}

\subsection{Unordered graphs}
Given an (unordered) graph $F$ and a natural number $n$ we write $\ex(n, F)$
for the maximal number of edges that a graph on $n$ vertices can have, if it
is $F$-free, i.e., if it has no subgraphs isomorphic to $F$. An averaging
argument shows that the
sequence $n\longmapsto \ex(n, F)/\binom n2$ is nonincreasing and, therefore,
the limit
\begin{equation}\label{eq:1614}
	\pi(F)=\lim_{n\to\infty}\frac{\ex(n, F)}{\binom n2}\,,
\end{equation}
known as the {\it Tur\'an density of $F$}, exists. Results of Erd\H{o}s,
Simonovits, and Stone~\cites{ErSi66, ErSt46} yield an exact formula for
this graph parameter. Specifically, if $F$ has at least one edge, then
\begin{equation}\label{eq:1632}
	\pi(F)=1-\frac 1{\chi(F)-1}\,,
\end{equation}
where~$\chi(F)$ denotes the chromatic number of $F$. With the exception of
the bipartite case, this gives a fairly complete picture.

One popular direction
of further study replaces the ambient host graph $K_n$, in which the extremal
$F$-free graphs are thought of as living and whose number of edges appears in
the denominator of~\eqref{eq:1614}, by other graphs. For instance, beginning
with the work of Kosto\v{c}ka~\cite{Ko76} people have been investigating Tur\'an problems
in hypercubes (see also~\cites{AKS, Ax}). Trying to optimise over the host graph,
however, is less interesting than it might appear at first. By averaging over all
permutations of $V(G)$ one can show that
\begin{equation}\label{eq:1654}
	\text{ \it
		every graph $G$ has an $F$-free subgraph $G'$ with at least $\pi(F)e(G)$
		edges. }
\end{equation}
Moreover, for every fixed graph $F$ this statement becomes false if we
replace~$\pi(F)$ by any larger constant (as can be seen by taking $G=K_n$
and letting $n$ tend to infinity).

\subsection{Ordered graphs}
From now on all graphs we consider will be {\it ordered}, that is they will be
equipped with a distinguished linear ordering of their vertex sets. Accordingly,
when we say that a graph $G$ contains another graph $F$ as a subgraph, this also
means that $F$ appears `correctly ordered' in $G$. Bearing this in mind, one can
define extremal numbers~$\wseq\ex(n, F)$ and Tur\'an densities $\vpi(F)$ as in the
unordered case. Research on these quantities was initiated
by Pach and Tardos~\cite{PT06}, who found the appropriate adaptation of~\eqref{eq:1632} to ordered graphs, namely
\begin{equation}\label{eq:1634}
	\vpi(F)=1-\frac 1{\chi_<(F)-1}\,.
\end{equation}
Here $\chi_<(F)$ denotes the so-called {\it interval chromatic number} of $F$,
defined to be the least number of colours required to colour $V(F)$ properly,
but with the additional constraint that every colour class needs to be an interval.

A notable example, where $\chi$ and $\chi_<$ differ significantly, is the
ascending path $P_k$ with $k$ edges, defined by $V(P_k)=[k+1]$ and
$E(P_k)=\{\{i, i+1\}\colon i\in [k]\}$. It is plain that $P_k$ is bipartite
and has Tur\'an density zero in the unordered sense; but due to
$\chi_<(P_k)=k+1$ we have
\[
	\vpi(P_k)=\frac{k-1}k\,.
\]

The averaging argument~\eqref{eq:1654} does not extend to ordered graphs,
because the class of (extremal) $F$-free graphs is, in general, not
closed under permutations of vertices. This leads us to a new interesting graph parameter.

\begin{dfn}\label{dfn:pim}
	Given an ordered graph $F$ we define~$\rho(F)$, called the {\it relative
			Tur\'an density of $F$}, to be the largest real number $\varsigma\in [0, 1]$
	with the following property: Every graph $G$ has an $F$-free subgraph $G'$
	with $e(G')\ge \varsigma e(G)$.
\end{dfn}
It follows from the definition that
\begin{equation}\label{eq:ub}
	\rho(F)\leq \vpi(F)
\end{equation}
for every ordered graph $F$ and equality holds for the ordered cliques $K_r$ on $r\geq 2$ vertices.
In the other direction we observe for any ordered graph $F$
\begin{equation}\label{eq:lb}
	\rho(F)\geq \frac{\l(F)-1}{2\l(F)}\,,
\end{equation}
where $\l(F)$ is number of edges of a longest monotone path in~$F$.
Given an ordered graph~$G$ we consider all maps $\phi\colon V(G)\lra [\l(F)]$.
For each of them we let~$G_\phi$ be the subgraph of~$G$ having all edges $xy$
with $x<y$ and $\phi(x)<\phi(y)$. Since there are no strictly monotone sequences
of length $\l(F)+1$ in $[\l(F)]$, all these graphs $G_\phi$ are $F$-free.
Moreover, on average~$G_\phi$ has
\[
	\frac{\l(F)-1}{2\l(F)}e(G)
\]
edges and thus there exists
some map~$\phi$ such that $G_\phi\subseteq G$ exemplifies the lower bound~\eqref{eq:lb}.

A result on shift graphs due to Arman, R\"odl, and Sales~\cite{ARS}
implies~$\rho(P_2)=\frac 14$, which shows that the lower  bound~\eqref{eq:lb} is optimal in this case.
Our main result generalises this to longer monotone paths and will be proved in \ssign\ref{sec:path}.

\begin{thm}\label{thm:main}
	We have $\rho(P_k)=\frac{k-1}{2k}$ for every $k\ge 2$.
\end{thm}

We shall also show that, like many other variants of the Tur\'an density,~$\rho$ is invariant under taking
blow-ups. Given an ordered graph $F$ and a positive integer $t$ we shall write~$F(t)$
for the ordered graph obtained from $F$ by replacing each vertex $x$ by an
interval~$I_x$ of size $t$ and every edge $xy$ by all $t^2$ edges from $I_x$ to $I_y$.
We also require that for all vertices $x<y$ of $F$, all vertices in $I_x$ precede all
vertices in $I_y$ with respect to the ordering of $F(t)$. A standard supersaturation argument carried out in \ssign\ref{sec:blow}
yields the following.

\begin{prop}\label{lem:1620}
	For all ordered graphs $F$ and integers $t\geq 1$ we have $\rho(F(t))=\rho(F)$.
\end{prop}

\section{Blow-ups}
\label{sec:blow}

\begin{proof}[Proof of Proposition~\ref{lem:1620}]
	The estimate $\rho(F(t))\ge \rho(F)$ being clear, we shall show that
	\[
		\rho(F(t))\le \rho(F)+2\eps
	\]
	holds for every given $\eps>0$. To this end
	we take a graph $G$ which has no $F$-free subgraph~$G'$ of
	size $e(G')\ge (\rho(F)+\eps)e(G)$. Suppose for the sake of notational
	simplicity that $V(G)=[m]$ holds for some natural number $m$.

	We contend that for a sufficiently large integer $n$ (relative to~$m$, $v(F)$, $t$,
	and~$\eps$) the blow-up $H=G(n)$
	exemplifies $\rho(F(t))\le \rho(F)+2\eps$. In other words, we shall prove
	that every subgraph $H'$ of $H$ of size $e(H')\ge (\rho(F)+2\eps)e(H)$
	contains a copy of $F(t)$.

	Let $I_1, \dots, I_m$ be the vertex classes of $H$. By a {\it transversal}
	we shall mean an $m$-element subset of $V(H)$ intersecting each of these classes
	exactly once. Denoting the set of all transversals by $\gT$ we have $|\gT|=n^m$
	and
	\[
		\sum_{T\in\gT}e_{H'}(T)
		=
		n^{m-2}e(H')
		\ge
		(\rho(F)+2\eps) n^{m-2} e(H)
		=
		(\rho(F)+2\eps) n^m e(G)\,.
	\]
	Consequently, the subset
	\[
		\gT_\star=\{T\in \gT\colon e_{H'}(T)\ge (\rho(F)+\eps) e(G)
	\]
	of `rich' transversals satisfies
	\[
		(\rho(F)+2\eps) e(G) n^m
		\le
		(\rho(F)+\eps) e(G) n^m
		+
		e(G) |\gT_\star|\,,
	\]
	whence $|\gT_\star|\ge \eps n^m$.

	By our choice of $G$, each rich transversal $T\in\gT_\star$ contains a copy of $F$
	which {\it crosses} the partition $\{I_1, \dots, I_m\}$, i.e., whose vertex set
	intersects each class $I_i$ at most once. Conversely, every crossing copy
	of~$F$ in~$H'$ belongs to at most $n^{m-v(F)}$ transversals.
	For these reasons, there are at least $\eps n^{v(F)}$ crossing copies of $F$
	in $H'$. Setting $f=v(F)$ this yields $f$ indices $m(1)<\dots<m(f)$ in $[m]$
	such that the $f$-partite subgraph of $H'$ induced by $V_{m(1)}, \dots, V_{m(f)}$
	contains at least $\eps n^f/\binom mf$ crossing copies of $F$. So the $f$-partite
	$f$-uniform hypergraph with these vertex classes whose edges correspond to the
	crossing copies of $F$ has positive density. By a result of Erd\H{o}s~\cite{Er64},
	a sufficiently large choice of $n$ guarantees that
	this hypergraph contains
	a complete $f$-partite hypergraph with vertex classes of size $t$.
	Consequently, $H'$ has indeed a subgraph isomorphic to~$F(t)$.
\end{proof}

\section{Paths}
\label{sec:path}

Throughout this section, which is devoted to the proof of Theorem~\ref{thm:main},
we fix an integer~$k\ge 2$.
We note that the lower bound follows from the general inequality~\eqref{eq:lb}.
The corresponding upper bound requires the construction of appropriate graphs $G$.
Before introducing those, we shall discuss a quadratic inequality, which we require
later. We write
\[
	\Delta_k=\{(\alpha_1, \dots, \alpha_k)\in [0, 1]^k\colon
	\alpha_1+\dots+\alpha_k=1\}
\]
for the $(k-1)$-dimensional standard simplex. We designate elements of~$\RR^k$ by
lowercase greek letters and the coordinates of any $\xi\in\RR^k$ will be denoted
by $\xi_1, \dots, \xi_k$. For every nonnegative integer~$d$ the function
$h_d\colon \Delta_k\lra\RR$ is defined by
\[
	h_d(\alpha)
	=
	(d+2)(1-\|\alpha\|^2)+k\sum_{r=1}^d\frac 1r\,,
\]
where $\|\cdot\|$ refers to the Euclidean standard norm.

\begin{lemma}\label{lem:hd-rek}
	If $\alpha, \beta, \gamma\in \Delta_k$ satisfy $2\alpha=\beta+\gamma$
	and $d\ge 1$, then
	\[
		h_{d-1}(\beta)+h_{d-1}(\gamma)+4\sum_{1\le i<j\le k}\beta_i\gamma_j
		\le
		2h_d(\alpha)\,.
	\]
\end{lemma}

\begin{proof}
	Set $\eta=\beta-\alpha=\alpha-\gamma$. The parallelogram law tells us
	$\|\beta\|^2+\|\gamma\|^2=2(\|\alpha\|^2+\|\eta\|^2)$ and
	thus we only need to show
	\[
		2\sum_{1\le i<j\le k}\beta_i\gamma_j
		\le
		(1-\|\alpha\|^2)+(d+1)\|\eta\|^2+\frac kd\,.
	\]
	The left side evaluates to
	\begin{equation}\label{eq:1648}
		2\sum_{1\le i<j\le k}(\alpha_i+\eta_i)(\alpha_j-\eta_j)
		=
		2\sum_{1\le i<j\le k}(\alpha_i\alpha_j-\eta_i\eta_j)
		+
		2\sum_{1\le i\le k}\lambda_i\eta_i\,,
	\end{equation}
	where $\lambda_i=\sum_{j>i}\alpha_j-\sum_{j<i}\alpha_j$
	satisfies $|\lambda_i|\le 1$ due to $\alpha\in\Delta_k$.
	Because of $\sum_i\alpha_i=1$ and $\sum_i \eta_i=0$
	the double sum on the right side of~\eqref{eq:1648}
	simplifies to $1-\|\alpha\|^2+\|\eta\|^2$.
	Therefore, it remains to prove
	\[
		2\sum_{i=1}^k \lambda_i\eta_i
		\le
		d\|\eta\|^2+\frac kd\,.
	\]
	But this is clearly implied by
	\[
		0
		\le
		d\sum_{i=1}^k (\eta_i-\lambda_i/d)^2
		\le
		d\|\eta\|^2-2\sum_{i=1}^k \lambda_i\eta_i+d^{-1}\sum_{i=1}^k\lambda_i^2\,.
		\qedhere
	\]
\end{proof}

Arman, R\"odl, and Sales describe in~\cite{ARS}*{Definition~5} a class of
graphs~$G_\eps(n, d)$ that we shall need as well. If a real number $\eps\in (0, 1]$
and a nonnegative integer $d$ are given, such graphs $G_\eps(n, d)$ exist for all
sufficiently large multiples $n$ of $2^d$. The construction proceeds by recursion
on $d$. To begin with, for every $n\in\NN$ we let $G_\eps(n, 0)$ be the empty graph
on~$[n]$ without any edges. Now suppose that all graphs of the form $G_\eps(n, d-1)$
have already been defined and let $n$ be a sufficiently large integer divisible
by $2^d$.
Fix a quasirandom bipartite graph $B=B_\eps(n, d)$ with vertex
classes $[n/2]$, $(n/2, n]$ and density $2^{1-d}$. More precisely, the demands
on this bipartite graph are $e(B)=n^2/2^{d+1}$ and
\[
	e_B(X, Y)=\frac{|X||Y|}{2^{d-1}}\pm \frac{\eps n^2}{k2^{d+2}}
\]
for all subsets $X\subseteq [n/2]$ and $Y\subseteq (n/2, n]$. (It is well known
that such bipartite graphs exist for all sufficiently large numbers $n$ divisible by $2^{d+1}$; we refer
to the appendix of~\cite{ARS} for the standard probabilistic proof.)
Having chosen $B_\eps(n, d)$ we define $G_\eps(n, d)$ such that
\begin{enumerate}
	\item[$\bullet$] its subgraphs induced by $[n/2]$ and $(n/2, n]$ are isomorphic
	      to $G_\eps(n/2, d-1)$
	\item[$\bullet$] and its bipartite subgraph between $[n/2]$ and $(n/2, n]$ is
	      isomorphic to $B_\eps(n, d)$.
\end{enumerate}

For instance, $G_\eps(n, 1)=B_\eps(n, 1)$ is the complete bipartite graph with
vertex classes~$[n/2]$ and $(n/2, n]$. An easy induction on $d$ discloses
\begin{equation}\label{eq:2345}
	e(G_\eps(n, d))
	=
	\frac{dn^2}{2^{d+1}}\,,
\end{equation}
whenever the graph $G_\eps(n, d)$ is defined (cf.\ \cite{ARS}*{Eq.~(4)}).

\begin{lemma}
	Given $\eps>0$ suppose that $n$ and $d$ are such that the graph $G_\eps(n, d)$
	exists. If $[n]=V_1\dcup\dots\dcup V_k$
	is a partition, and $\alpha_i=|V_i|/n$ for every $i\in [k]$, then the number
	of edges~$xy$ of $G_\eps(n, d)$ with $x<y$ and $x\in V_i$, $y\in V_j$ for
	some $i<j$ is at most $(h_d(\alpha)+d\eps)n^2/2^{d+2}$,
	where $\alpha=(\alpha_1, \dots, \alpha_k)$.
\end{lemma}

\begin{proof}
	We argue by induction on $d$. The base case, $d=0$, is clear,
	because $G_\eps(n, 0)$ has no edges at all. Now consider the induction
	step from $d-1$ to $d$ and let $|V_i\cap [1, n/2]|=\beta_i n/2$ as well
	as $|V_i\cap (n/2, n]|=\gamma_i n/2$ for every $i\in [k]$. Clearly the
	vectors $\beta=(\beta_1, \dots, \beta_k)$ and $\gamma=(\gamma_1, \dots, \gamma_k)$
	are in $\Delta_k$ and satisfy $2\alpha=\beta+\gamma$.
	There are three kinds of edges to consider:
	\begin{enumerate}[label=\alabel]
		\item\label{it:e1} those with $x, y\in [1, n/2]$,
		\item\label{it:e2} those with $x, y\in (n/2, n]$,
		\item\label{it:e3} and those with $1\le x\le n/2<y\le n$.
	\end{enumerate}

	By the induction hypothesis there are at most
	\[
		\frac{(h_{d-1}(\beta)+(d-1)\eps)(n/2)^2}{2^{d+1}}
		\quad \text{ and } \quad
		\frac{(h_{d-1}(\gamma)+(d-1)\eps)(n/2)^2}{2^{d+1}}
	\]
	edges of types~\ref{it:e1} and~\ref{it:e2}, respectively.
	Moreover, by quasirandomness, there are at most
	\[
		\sum_{i=1}^{k}\Big(\frac{\beta_i\sum_{j>i}\gamma_j}{2^{d-1}}\Big(\frac{n}{2}\Big)^2+\frac{\eps n^2}{k2^{d+2}}\Big)
		=
		\frac{\big(4\sum_{i<j}\beta_i\gamma_j +2\eps\big)n^2}{2^{d+3}}
	\]
	edges of type~\ref{it:e3}.
	Altogether, the number $\Omega$ of edges under consideration satisfies
	\[
		\frac{\Omega}{n^2/2^{d+3}}
		\le
		h_{d-1}(\beta)+h_{d-1}(\gamma)+4\sum_{i<j}\beta_i\gamma_j+2d\eps\,,
	\]
	and by Lemma~\ref{lem:hd-rek} the right side is at most $2h_d(\alpha)+2d\eps$.
\end{proof}

Now the upper bound $\rho(P_k)\le\frac{k-1}{2k}$ we still seek to establish
is a straightforward consequence of the following result.

\begin{lemma}
	For every $\eps>0$ there are positive integers $d$ and $n$ such that
	$G=G_\eps(n, d)$ is defined and every $P_k$-free subgraph $G'$ of $G$
	satisfies $e(G')\le (\frac{k-1}{2k}+\eps)e(G)$.
\end{lemma}

\begin{proof}
	Since $\sum_{r=1}^d 1/r=\log d+O(1)=o(d)$, we can choose $d$ so large that
	\begin{equation}\label{eq:2305}
		2+k \sum_{r=1}^d \frac 1r
		\le
		\eps d\,.
	\end{equation}
	Let $n$ be an arbitrary number for which the graph $G=G_\eps(n, d)$ is defined
	and consider any $P_k$-free subgraph $G'$ of $G$.

	For each vertex $x\in [n]$ let $f(x)$ be the largest positive integer such
	that~$G'$ contains an ascending path of length $f(x)-1$ ending in $x$.
	By our assumption on $G'$, this function only attains values in $[k]$.
	Moreover, if $xy$ with $x<y$ is an edge in $G'$, then $f(x)<f(y)$.
	Thus, setting $\alpha_i=|f^{-1}(i)|/n$ for every $i\in [k]$ and
	$\alpha=(\alpha_1, \dots, \alpha_k)$, the previous lemma and~\eqref{eq:2345}
	yield
	\[
		\frac{e(G')}{e(G)}
		\le
		\frac{h_d(\alpha)+d\eps}{2d}\,.
	\]
	Due to $\|\alpha\|^2\ge 1/k$ we also have
	\[
		h_d(\alpha)
		\le
		\frac{d(k-1)}k+2+k\sum_{r=1}^d \frac1r
		\overset{\eqref{eq:2305}}{\le}
		d\left(\frac{k-1}k+\eps\right)\,,
	\]
	which leads indeed to $e(G')\le (\frac{k-1}{2k}+\eps)e(G)$.
\end{proof}

This completes the proof of Theorem~\ref{thm:main}.

\section{Concluding remarks}
The case $k=3$ of Theorem~\ref{thm:main} yields a positive answer
to~\cite{ARS}*{Problem~9}. Here we discuss a few further problems
for future research.

Let $C_\ell$ denote the {\it ordered cycle} with vertex set $V(C_\ell)=[\ell]$
and edge set defined by~$E(C_\ell)=\{\{i, i+1\}\colon i\in [\ell-1]\}\dcup \{\{1, \ell\}\}$.
Since $C_\ell$ contains a copy of the monotone path $P_{\ell-1}$, Theorem~\ref{thm:main}
yields $\rho(C_\ell)\ge\frac{\ell-2}{2\ell-2}$ leading to the following problem.
\begin{prob}
	Determine $\rho(C_\ell)$ for every fixed $\ell\ge 4$
\end{prob}

The obvious inequality~\eqref{eq:ub} suggests the next question.

\begin{prob}
	Characterise the class $\{F\colon \rho(F)=\vpi(F)\}$.
\end{prob}

For instance, all ordered cliques $K_r$ are in this class. Are there any other
such graphs? Similarly, when $F=P_k$ is a path, then Theorem~\ref{thm:main}
yields $\rho(F)=\frac12\vpi(F)$; we may thus ask for a characterisation of the
class $\{F\colon \rho(F)=\frac12\vpi(F)\}$. Are there any graphs $F$ satisfying
\[
	\frac12\vpi(F)<\rho(F)<\vpi(F)\,\text{?}
\]

It would also be interesting to know whether there is any stability result
accompanying Theorem~\ref{thm:main}.

\begin{prob}
	Given $k\ge 2$ and $\eps>0$, describe the structure of all graphs $G$ with the
	property that every subgraph $G'$ satisfying $e(G')\ge(\frac{k-1}{2k}+\eps)e(G)$
	contains a copy of $P_k$.
\end{prob}

In particular, one may ask whether such graphs need to have
any resemblance to~$G_\eta(n, d)$ for some small $\eta=\eta(k, \eps)$.
Perhaps one should also
assume here that $G$ be dense, i.e., that $e(G)\ge \eps v(G)^2$.

Finally, the definition of $\rho(\cdot)$ generalises straightforwardly to
hypergraphs. A special case studied by Erd\H{o}s, Hajnal, and Szemer\'edi~\cite{EHS}
in the context of independent sets in shift graphs  concerns the
ascending $r$-uniform path $P^{(r)}_2$ of length $2$, i.e., the hypergraph
on~$[r+1]$ with edges $[r]$ and $\{2, \dots, r+1\}$. In the current notation,
they showed
\[
	\rho(P^{(r)}_2)
	\ge
	\begin{cases}
		\frac12-\frac1r    & \text{ if $r$ is even} \\
		\frac12-\frac1{2r} & \text{ if $r$ is odd.}
	\end{cases}
\]
For $r=4$ the quantitative improvement $\rho(P^{(4)}_2)\ge \frac38$
was obtained in~\cite{ARS} (see the footnote on page~9).

\begin{prob}
	Determine $\rho(P^{(r)}_2)$ for all $r\ge 3$. In particular, is
	${\rho(P^{(3)}_2)=\frac13}$ true?
\end{prob}

Of course, one may also ask the same question for longer paths.

\end{document}